\documentclass[a4paper,12pt,intlimits,oneside]{amsart}

\usepackage{enumerate}
\usepackage{mathrsfs}
\usepackage{amsfonts,amsmath}

\usepackage{latexsym,amssymb,amsthm}

\usepackage[bookmarksnumbered, colorlinks, plainpages]{hyperref}
\hypersetup{colorlinks=true,linkcolor=red, anchorcolor=greenF, citecolor=cyan, urlcolor=red, filecolor=magenta, pdftoolbar=true}

\newcommand{\comment}[1]{}

\newcommand{\R}{{\mathbb R}}

\def\H{{\mathcal H}}
\def\Ha{{\mathscr H}}

\def\Hd{{\H^p_a(\mathbb C_+)}}
\def\Hau{{\mathscr H_\varphi}}

\newcounter{rea}
\newcounter{rek}
\newcounter{res}
\begin{document}
\title[]{Hausdorff operators on holomorphic Hardy spaces and applications}         
\author{Ha Duy Hung}    
\address{High School for Gifted Students, Hanoi National University of Education, 136 Xuan Thuy, Hanoi, Vietnam} 
\email{{\tt hunghaduy@gmail.com}}
\author{Luong Dang Ky}
\address{Department of Mathematics, Quy Nhon University, 
170 An Duong Vuong, Quy Nhon, Binh Dinh, Viet Nam} 
\email{{\tt luongdangky@qnu.edu.vn}}
\author{Thai Thuan Quang}
\address{Department of Mathematics, Quy Nhon University, 
170 An Duong Vuong, Quy Nhon, Binh Dinh, Viet Nam} 
\email{thaithuanquang@qnu.edu.vn}
\thanks{This research is funded by Vietnam National Foundation for Science and Technology Development (NAFOSTED) under grant number 101.02-2016.22.}

\begin{abstract} The aim of this paper is to characterize the nonnegative functions $\varphi$ defined on $(0,\infty)$ for which the Hausdorff operator 
	$$\Ha_\varphi f(z)= \int_0^\infty f\left(\frac{z}{t}\right)\frac{\varphi(t)}{t}dt$$
is bounded on the Hardy spaces of the upper half-plane $\mathcal H_a^p(\mathbb C_+)$, $p\in[1,\infty]$. The corresponding operator norms and their applications  are also given.
\end{abstract}

\maketitle
\newtheorem{theorem}{Theorem}[section]
\newtheorem{lemma}{Lemma}[section]
\newtheorem{proposition}{Proposition}[section]
\newtheorem{remark}{Remark}[section]
\newtheorem{corollary}{Corollary}[section]
\newtheorem{definition}{Definition}[section]
\newtheorem{example}{Example}[section]
\numberwithin{equation}{section}
\newtheorem{Theorem}{Theorem}[section]
\newtheorem{Lemma}{Lemma}[section]
\newtheorem{Proposition}{Proposition}[section]
\newtheorem{Remark}{Remark}[section]
\newtheorem{Corollary}{Corollary}[section]
\newtheorem{Definition}{Definition}[section]
\newtheorem{Example}{Example}[section]
\newtheorem*{theorema}{Theorem A}

\section{Introduction and the main result} 
\allowdisplaybreaks

Let $\varphi$ be a locally integrable function on $(0,\infty)$. The {\it Hausdorff operator} $H_\varphi$ is then defined for suitable functions $f$ on $\R$ by
\begin{equation}\label{real Hausdorff operator}
H_\varphi f(x)=\int_0^\infty f\left(\frac{x}{t}\right) \frac{\varphi(t)}{t} dt,\quad x\in\R.	
\end{equation}

The Hausdorff operator is an interesting operator in harmonic analysis. There are many classical operators in analysis which are special cases of the Hausdorff operator if one chooses suitable kernel functions $\varphi$, such as the classical Hardy operator, its adjoint operator, the Ces\`aro type operators, the Riemann-Liouville fractional integral operator,... See the survey article \cite{Li} and the references therein. In the recent years, there is an increasing interest in the study of boundedness of the Hausdorff operator and its commuting with the Hilbert transform on the real Hardy spaces and on the Lebesgue spaces, see for example \cite{An, AS, BBMM, BHS, HKQ, Li07, Li, LM, LM2, Xi}.

Let $\mathbb C_+$ be the upper half-plane in the complex plane. For $0<p\leq \infty$, the Hardy space $\mathcal H_a^p(\mathbb C_+)$ is defined as the set of all holomorphic functions $f$ on $\mathbb C_+$ such that
$$\|f\|_{\mathcal H_a^p(\mathbb C_+)}:= \sup_{y>0} \left(\int_{-\infty}^{\infty} |f(x+iy)|^p dx\right)^{1/p}<\infty$$
if $0<p<\infty$, and if $p=\infty$, then
$$\|f\|_{\mathcal H_a^\infty(\mathbb C_+)}:= \sup_{z\in\mathbb C_+} |f(z)|<\infty.$$

It is classical (see \cite{Du, Ga}) that if $f\in \H_a^p(\mathbb C_+)$, then $f$ has a {\it boundary value function} $f^*\in L^p(\R)$ defined by
$$f^*(x)= \lim_{y\to 0} f(x+iy),\quad \mbox{a.e.}\; x\in\R.$$

Let $p\in [1,\infty]$ and let $\varphi$ be a nonnegative function in $L^1_{\rm loc}(0,\infty)$ for which
\begin{equation}\label{main inequality}
\int_0^\infty t^{1/p-1}\varphi(t)dt<\infty.
\end{equation}
Then it is well-known (see \cite{An}) that $H_\varphi$ is bounded on $L^p(\R)$, and thus $H_\varphi(f^*)\in L^p(\R)$ for any boundary value function $f^*$ of a function $f$ in $\H_a^p(\mathbb C_+)$. A natural question arises is that whether the transformed function $H_\varphi(f^*)$ is also the boundary value function of a function in $\H_a^p(\mathbb C_+)$? In some special cases of $\varphi$ and $1<p<\infty$, using the spectral mapping theorem and the Hille-Yosida-Phillips theorem, Arvanitidis-Siskakis \cite{AS} and  Ballamoole-Bonyo-Miller-Millerstudied \cite{BBMM} studied and gave affirmative answers to this question.

In the present paper, we give an affirmative answer to the above question by studying a complex version of $H_\varphi$ defined by
$$\Ha_\varphi f(z)= \int_0^\infty f\left(\frac{z}{t}\right)\frac{\varphi(t)}{t}dt,\quad z\in \mathbb C_+.$$

Our main result reads as follows.

\begin{theorem}\label{main theorem}   
	Let $p\in [1,\infty]$ and let $\varphi$ be a nonnegative function in $L^1_{\rm loc}(0,\infty)$. Then $\Ha_\varphi$ is bounded on $\H_a^p(\mathbb C_+)$ if and only if (\ref{main inequality}) holds. Moreover, in that case, we obtain
 $$
 	\|\Ha_\varphi\|_{\H^p_a(\mathbb C_+)\to \H^p_a(\mathbb C_+)}=\int_0^\infty t^{1/p-1}\varphi(t)dt
 $$
 and, for any $f\in \H_a^p(\mathbb C_+)$,
 $$
 (\Ha_\varphi f)^*= H_\varphi(f^*).
 $$
\end{theorem}

It should be pointed out that some main results in \cite{AS, BBMM} (see \cite[Theorems 3.1, 3.3 and 4.1]{AS} and \cite[Theorem 3.4]{BBMM}) can be viewed as special cases of Theorem \ref{main theorem} by choosing suitable kernel functions $\varphi$. In the setting of Hardy spaces $\mathcal H^p(\mathbb D)$ on the unit disk $\mathbb D=\{z\in \mathbb C: |z|<1\}$, Galanopoulos and Papadimitrakis (\cite[Theorems 2.3 and 2.4]{GP}) studied and obtained some similar results to Theorem \ref{main theorem} for $1<p<\infty$ while it is slightly different at the endpoints $p=1$ and $p=\infty$  (see also the survey article \cite{Li}).

Furthermore, if we denote by $\H^1(\R)$ the real Hardy space in the sense of Fefferman-Stein (see the last section), then by using Theorem \ref{main theorem}, we obtain the following result.

\begin{corollary}[see Theorem \ref{a lower bound}]
	Let $\varphi$ be a nonnegative function in $L^1_{\rm loc}(0,\infty)$ such that $H_\varphi$ is bounded on $\H^1(\R)$. Then, 
	$$\int_0^\infty \varphi(t)dt \leq \|H_\varphi\|_{\H^1(\R)\to \H^1(\R)}<\infty.$$
\end{corollary}

The above corollary is not only give an answer to a question posted by Liflyand \cite[Problem 4]{Li07}, but also give a lower bound for the norm of $H_\varphi$ on $\H^1(\R)$. Another corollary of Theorem \ref{main theorem} is:

\begin{corollary}[see Theorem \ref{commutes with the Hilbert transform}]
	Let $p\in (1,\infty)$ and let $\varphi$ be as in Theorem \ref{main theorem}. Then $H_\varphi$ is bounded on $L^p(\R)$ if and only if (\ref{main inequality}) holds. Moreover, in that case, 
	$$\|H_\varphi\|_{L^p(\R)\to L^p(\R)}= \int_0^\infty t^{1/p-1}\varphi(t)dt$$
	and $H_\varphi$ commutes with the Hilbert transform $H$ on $L^p(\R)$. 
\end{corollary}

Throughout the whole article, we use the symbol $A \lesssim B$ (or $B\gtrsim A$) means that $A\leq C B$ where $C$ is a positive constant which is independent of the main parameters, but it may vary from line to line. If $A \lesssim B$ and $B\lesssim A$, then we  write $A\sim B$.  For any $E\subset \R$, we denote by $\chi_E$ its characteristic function.


\section{Proof of Theorem \ref{main theorem}}

In the sequel, we always assume that $\varphi$ is a nonnegative function in $L^1_{\rm loc}(0,\infty)$. Also we remark that, for any $f\in \mathcal H_a^p(\mathbb C_+)$, the function $\Hau f$ is well-defined and  holomorphic on $\mathbb C_+$ provided (\ref{main inequality}) holds, since 
\begin{equation}\label{a pointwise estimate for f}
|f(x+iy)|\leq \left(\frac{2}{\pi y}\right)^{1/p} \|f\|_{\mathcal H_a^p(\mathbb C_+)}
\end{equation}
and 
\begin{equation}
|f'(x+iy)| \leq \frac{2}{y}\left(\frac{4}{\pi y}\right)^{1/p}\|f\|_{\mathcal H_a^p(\mathbb C_+)}
\end{equation}
for all $z=x+iy\in\mathbb C_+$. See Garnett's book \cite[p. 57]{Ga}.

Given an holomorphic function $f$ on $\mathbb C_+$, we define the {\it nontangential maximal function} of $f$ by
$$\mathcal M(f)(x)= \sup_{|t-x|<y} |f(t+iy)|, \quad x\in \R.$$

The following lemma is classical and can be found in \cite{Du, Ga}.

\begin{lemma}\label{nontangential maximal function characterization}
	Let $0<p<\infty$. Then:
	\begin{enumerate}[\rm (i)]
		\item For any $f\in \mathcal H_a^p(\mathbb C_+)$, we have
		$$\|f^*\|_{L^p(\R)}= \|f\|_{\mathcal H_a^p(\mathbb C_+)}\quad\mbox{and}\quad \lim_{y\to 0}\|f(\cdot+iy)- f^*(\cdot)\|_{L^p(\R)}=0.$$
		\item $f\in \mathcal H_a^p(\mathbb C_+)$ if and only if $\mathcal M(f)\in L^p(\R)$. Moreover,
		$$\|f\|_{\mathcal H_a^p(\mathbb C_+)}\sim \|\mathcal M(f)\|_{L^p(\R)}.$$		
	\end{enumerate}	
\end{lemma}

\begin{lemma}\label{a lemma for p is infinity}
	Theorem \ref{main theorem} is true for $p=\infty$.
\end{lemma}

\begin{proof}
	Suppose that $\int_0^\infty t^{-1}\varphi(t)dt$ is finite. Then, for any $f\in \H^\infty_a(\mathbb C_+)$, 
	$$\|\Ha_\varphi f\|_{\mathcal H_a^\infty(\mathbb C_+)}=\sup_{z\in\mathbb C_+}\left|\int_0^\infty f\left(\frac{z}{t}\right) \frac{\varphi(t)}{t} dt\right|\leq \int_0^\infty t^{-1}\varphi(t)dt \|f\|_{\mathcal H^\infty_a(\mathbb C_+)}.$$
	Therefore, $\Ha_\varphi$ is bounded on $\mathcal H^\infty_a(\mathbb C_+)$, moreover,
	\begin{equation}\label{a lemma for p is infinity, 1}
		\|\Hau\|_{\H^\infty_a(\mathbb C_+)\to \H^\infty_a(\mathbb C_+)}\leq \int_0^\infty t^{-1}\varphi(t)dt.
	\end{equation}
		
	On the other hand, we have
	$$\|\Hau\|_{\H^\infty_a(\mathbb C_+)\to \H^\infty_a(\mathbb C_+)}\geq \frac{\|\Hau(1)\|_{\H^\infty_a(\mathbb C_+)}}{\|1\|_{\H^\infty_a(\mathbb C_+)}}=\int_0^\infty t^{-1}\varphi(t)dt.$$
	This, together with (\ref{a lemma for p is infinity, 1}), implies that
	$$\|\Hau\|_{\H^\infty_a(\mathbb C_+)\to \H^\infty_a(\mathbb C_+)}= \int_0^\infty t^{-1}\varphi(t)dt.$$
	Moreover, by the dominated convergence theorem, for any $x\ne 0$,
	$$(\Ha_\varphi f)^*(x)=\lim_{y\to 0}\int_0^\infty f\left(\frac{x}{t}+ \frac{y}{t}i\right) \frac{\varphi(t)}{t}dt = \int_0^1 f^*\left(\frac{x}{t}\right) \frac{\varphi(t)}{t}dt = H_\varphi(f^*)(x).$$

	Conversely, suppose that $\Hau$ is bounded on $\H^\infty_a(\mathbb C_+)$. As the function $f(z)\equiv 1$ is in  $\H^\infty_a(\mathbb C_+)$, we obtain that $\Hau f= \int_0^1 t^{-1}\varphi(t)dt<\infty$.  
\end{proof}

\begin{lemma}\label{key lemma}
	Let $p\in [1,\infty)$ and let $\varphi$ be such that (\ref{main inequality}) holds. Then 
	\begin{enumerate}[\rm (i)]
		\item $\Hau$ is bounded on $\Hd$, moreover,
		$$\|\Hau\|_{\Hd\to \Hd}\leq \int_0^\infty t^{1/p-1}\varphi(t)dt.$$
		\item If supp $\varphi\subset [0,1]$, then
		$$\|\Hau\|_{\Hd\to \Hd} = \int_0^1 t^{1/p-1}\varphi(t)dt.$$
		\item For any $f\in \Hd$, we have
		$$(\Ha_\varphi f)^*= H_\varphi(f^*).$$		
	\end{enumerate}		
\end{lemma}

\begin{proof}
	(i) For any $f\in \mathcal H^p_a(\mathbb C_+)$, we have
	\begin{eqnarray*}
		\mathcal M(\Ha_\varphi f)(x) &=& \sup_{|u-x|<y} \left|\int_0^\infty f\left(\frac{u+iy}{t}\right)\frac{\varphi(t)}{t}dt\right|\\
		&\leq& \int_0^\infty \sup_{|\frac{u}{t}-\frac{x}{t}|<\frac{y}{t}} \left|f\left(\frac{u}{t}+\frac{y}{t}i\right)\right|\frac{\varphi(t)}{t}dt = H_{\varphi}(\mathcal M f)(x)
	\end{eqnarray*}	
	for all $x\in \mathbb R$. Therefore, by the Minkowski inequality and Lemma \ref{nontangential maximal function characterization}(ii),
	\begin{eqnarray*}
		\|\Ha_\varphi f\|_{\mathcal H^p_a(\mathbb C_+)}\lesssim  \|\mathcal M(\Ha_\varphi f)\|_{L^p(\mathbb R)} &\leq&  \|H_{\varphi}(\mathcal M f)\|_{L^p(\mathbb R)}\\
		&\leq& \int_0^\infty \left(\int_{\R} \left|\mathcal M f\left(\frac{x}{t}\right)\right|^p dx\right)^{1/p} \frac{\varphi(t)}{t} dt \\
		&=& \int_0^\infty t^{1/p-1}\varphi(t) dt \|\mathcal M f\|_{L^p(\mathbb R)}\\
		&\lesssim&  \int_0^\infty t^{1/p-1}\varphi(t) dt \|f\|_{\mathcal H^p_a(\mathbb C_+)}.
	\end{eqnarray*}
	This proves that $\Hau$ is bounded on $\mathcal H^p_a(\mathbb C_+)$, moreover,
	\begin{equation}\label{key lemma, 0}
		\|\Hau\|_{\Hd\to \Hd}\lesssim \int_0^\infty t^{1/p-1}\varphi(t) dt.
	\end{equation}

	In order to show 
	\begin{equation}\label{key lemma, 1}
	\|\Hau\|_{\Hd\to \Hd}\leq \int_0^\infty t^{1/p-1}\varphi(t) dt,	
	\end{equation}
	let us  first assume that (iii) is proved. Then, by Lemma \ref{nontangential maximal function characterization}(i) and the Minkowski inequality, we get
	\begin{eqnarray*}
		\left\|\Hau f\right\|_{\mathcal H_a^p(\mathbb C_+)} = \left\|(\Hau f)^*\right\|_{L^p(\mathbb R)} &=&\left\|H_\varphi (f^* )\right\|_{L^p(\mathbb R)}\\
		&\leq& \int_0^\infty \left(\int_{\R} \left|f^*\left(\frac{x}{t}\right)\right|^p dx\right)^{1/p} \frac{\varphi(t)}{t} dt\\
		&=& \|f^*\|_{L^p(\mathbb R)} \int_0^\infty t^{1/p-1}\varphi(t)dt \\
		&=& \|f\|_{\mathcal H_a^p(\mathbb C_+)} \int_0^\infty t^{1/p-1}\varphi(t)dt.
	\end{eqnarray*}
	This proves that (\ref{key lemma, 1}) holds.
	
	\vskip 0.3cm
	
(ii)	Let $\delta\in (0,1)$ be arbitrary and let $\varphi_\delta(t)= \varphi(t)\chi_{[\delta,\infty)}(t)$ for all $t\in (0,\infty)$. Since (\ref{key lemma, 1}) holds, we see that
$$
\|\mathscr H_{\varphi_\delta}\|_{\mathcal H^p_a(\mathbb C_+) \to \mathcal H^p_a(\mathbb C_+)} \leq \int_0^\infty t^{1/p-1}\varphi_\delta(t)dt = \int_\delta^1 t^{1/p-1}\varphi(t)dt  <\infty
$$
and
\begin{equation}\label{the truncation}
\|\Hau -\mathscr H_{\varphi_\delta}\|_{\mathcal H^p_a(\mathbb C_+) \to \mathcal H^p_a(\mathbb C_+)} \leq \int_0^\infty t^{1/p-1}[\varphi(t)-\varphi_\delta(t)] dt= \int_0^\delta t^{1/p-1}\varphi(t)dt.
\end{equation}

For any $\varepsilon>0$, we define the  function $f_\varepsilon: \mathbb C_+ \to\mathbb C$ by
$$f_\varepsilon(z)= \frac{1}{(z+i)^{1/p+\varepsilon}},$$
where, and in what follows, $\zeta^{1/p+\varepsilon}=|\zeta|^{1/p+\varepsilon} e^{i(1/p+\varepsilon)\arg \zeta}$ for all $\zeta\in\mathbb C$. Then
\begin{equation}\label{norm of f}
\|f_\varepsilon\|_{\mathcal H^p_a(\mathbb C_+)}= \left(\int_{-\infty}^{\infty}\frac{1}{\sqrt{x^2+1}^{1+p\varepsilon}} dx\right)^{1/p}<\infty.
\end{equation}

For all $z= x+iy\in\mathbb C_+$, we have
$$\mathscr H_{\varphi_\delta}(f_\varepsilon)(z) - f_\varepsilon(z)\int_0^\infty t^{1/p-1}\varphi_\delta(t)dt = \int_{\delta}^{1}[\phi_{\varepsilon,z}(t)-\phi_{\varepsilon,z}(1)] t^{1/p-1} \varphi(t)dt,$$
where $\phi_{\varepsilon,z}(t):=  \frac{t^{\varepsilon}}{(z+ ti)^{1/p+\varepsilon}}$. For any $t\in [\delta,1]$, a simple calculus gives
\begin{eqnarray*}
	|\phi_{\varepsilon,z}(t)-\phi_{\varepsilon,z}(1)| &\leq& |t-1|\sup_{s\in [\delta,1]}|\phi_{\varepsilon,z}'(s)|\\
	&\leq& \frac{\varepsilon \delta^{-1-1/p}}{\sqrt{x^2+ 1}^{1/p+\varepsilon}} + \frac{(1/p+\varepsilon) \delta^{-1-1/p}}{\sqrt{x^2+1}^{1+1/p+\varepsilon}}.
\end{eqnarray*}
This, together with (\ref{norm of f}),  yields
\begin{eqnarray}\label{an estimate for the norm}
	&&\frac{\left\|\mathscr H_{\varphi_\delta}(f_\varepsilon) - f_\varepsilon \int_0^\infty t^{1/p-1}\varphi_\delta(t)dt \right\|_{\mathcal H^p_a(\mathbb C_+)}}{\|f_\varepsilon\|_{\mathcal H^p_a(\mathbb C_+)}}\\
	&\leq& \int_{\delta}^{1} t^{1/p-1}\varphi(t)dt\left[\varepsilon \delta^{-1-1/p}+ \frac{(1/p+\varepsilon) \delta^{-1-1/p}\left(\int_{-\infty}^{\infty} \frac{1}{\sqrt{x^2+1}^{p+1}}\right)^{1/p}}{\left(\int_{-\infty}^{\infty}\frac{1}{\sqrt{x^2+1}^{1+p\varepsilon}}dx\right)^{1/p}}\right] \to 0\nonumber
\end{eqnarray}
as $\varepsilon\to 0$. As a consequence,
$$\int_{\delta}^{1} t^{1/p-1}\varphi(t)dt=\int_0^\infty t^{1/p-1}\varphi_\delta(t)dt \leq \|\mathscr H_{\varphi_\delta}\|_{\mathcal H^p_a(\mathbb C_+)\to \mathcal H^p_a(\mathbb C_+)}.$$
This, combined with (\ref{the truncation}), allows us to conclude that 
$$\|\Hau\|_{\mathcal H^p_a(\mathbb C_+)\to \mathcal H^p_a(\mathbb C_+)}\geq \int_0^1 t^{1/p-1}\varphi(t)dt - 2 \int_0^\delta t^{1/p-1}\varphi(t)dt\to \int_0^1 t^{1/p-1}\varphi(t)dt$$
as $\delta \to 0$ since $\int_0^1 t^{1/p-1}\varphi(t)dt<\infty$. Hence, by (\ref{key lemma, 1}),
$$\|\Hau\|_{\mathcal H^p_a(\mathbb C_+)\to \mathcal H^p_a(\mathbb C_+)}=  \int_0^1 t^{1/p-1}\varphi(t)dt.$$
	
(iii) For any  $\sigma>0$, it follows from (\ref{a pointwise estimate for f}) that the function
$$f_\sigma(z):= f(z+  i\sigma)$$
is in $\mathcal H_a^p(\mathbb C_+)\cap \mathcal H^\infty_a(\mathbb C_+)$. Let $\delta$ and $\varphi_\delta$ be as in (ii). Noting that
	$$\int_0^\infty t^{-1}\varphi_\delta(t)dt\leq \delta^{-1/p}\int_{\delta}^{\infty}t^{1/p-1}\varphi(t)dt<\infty,$$
Lemma \ref{a lemma for p is infinity}(ii) gives $(\mathscr H_{\varphi_\delta}(f_\sigma))^*= H_{\varphi_\delta}(f_\sigma^*)$. Therefore, by Lemma \ref{nontangential maximal function characterization}(i), \cite[Theorem 1]{An} and (\ref{key lemma, 0}), we obtain that  
	\begin{eqnarray*}
		&&\|(\Hau f)^* - H_\varphi(f^*)\|_{L^p(\R)}\\
		&\leq& \|\mathscr H_{\varphi-\varphi_\delta} f\|_{\Hd} + \|H_{\varphi-\varphi_\delta}(f^*)\|_{L^p(\R)} + \\
		&&  +\|\mathscr H_{\varphi_\delta}(f- f_\sigma)\|_{\Hd} + \|H_{\varphi_\delta}(f^* - f_\sigma^*)\|_{L^p(\R)}\\
		&\lesssim& \|f\|_{\Hd} \int_0^\infty t^{1/p-1}[\varphi(t)-\varphi_\delta(t)]dt + \|f^*-f_\sigma^*\|_{L^p(\R)}\int_0^\infty t^{1/p-1}\varphi_\delta(t)dt\\
		&\lesssim& \|f\|_{\Hd} \int_0^\delta t^{1/p-1}\varphi(t)dt + \|f^*-f_\sigma^*\|_{L^p(\R)}\int_0^\infty t^{1/p-1}\varphi(t)dt\to 0
	\end{eqnarray*}
	as $\sigma\to 0$ and $\delta\to 0$. This completes the proof of Lemma \ref{key lemma}.
		
\end{proof}

Now we are ready to give the proof of Theorem \ref{main theorem}.

\begin{proof}[Proof of Theorem \ref{main theorem}]
	By Lemmas \ref{a lemma for p is infinity} and \ref{key lemma}, it suffices to prove that
	\begin{equation}\label{main theorem, 0}
		\int_0^\infty t^{1/p-1} \varphi(t)dt \leq \|\Hau\|_{\Hd \to \Hd}
	\end{equation}
	whenever $\Hau$ is bounded on $\Hd$ for $1\leq p<\infty$.
		
	 Indeed, we first claim that
\begin{equation}\label{necessary condition}
\int_0^\infty t^{1/p-1} \varphi(t)dt <\infty.
\end{equation}

Assume (\ref{necessary condition}) holds for a moment. 

For any $m>0$, we define $\varphi_m(t)= \varphi(mt)\chi_{(0,1]}(t)$ for all $t\in (0,\infty)$. Then, by Lemma \ref{key lemma}(i), we see that
\begin{eqnarray}\label{main theorem, 1}
	\left\|\Hau -\mathscr H_{\varphi_m\left(\frac{\cdot}{m}\right)}\right\|_{\Hd \to \Hd} &=& \left\|\mathscr H_{\varphi-\varphi_m\left(\frac{\cdot}{m}\right)}\right\|_{\Hd \to \Hd}\nonumber\\
	&\leq& \int_0^\infty t^{1/p-1} \left[\varphi(t)- \varphi_m\left(\frac{t}{m}\right)\right] dt \nonumber\\
	&=& \int_m^\infty t^{1/p-1}\varphi(t)dt<\infty.
\end{eqnarray}

Noting that
$$\left\|f\left(\frac{\cdot}{m}\right)\right\|_{\Hd}= m^{1/p} \|f(\cdot)\|_{\Hd}\quad\mbox{and}\quad \mathscr H_{\varphi_m\left(\frac{\cdot}{m}\right)} f= \mathscr H_{\varphi_m} f\left(\frac{\cdot}{m}\right)$$
for all $f\in\Hd$, Lemma \ref{key lemma}(ii) gives
\begin{eqnarray*}
	\left\|\mathscr H_{\varphi_m\left(\frac{\cdot}{m}\right)}\right\|_{\Hd \to \Hd} &=& m^{1/p} \left\|\mathscr H_{\varphi_m}\right\|_{\Hd \to \Hd}\\
	&=& m^{1/p} \int_0^1 t^{1/p-1} \varphi_m(t)dt = \int_0^m t^{1/p-1} \varphi(t)dt.
\end{eqnarray*}
Combining this with (\ref{main theorem, 1}) allows us to conclude that
$$\|\Hau\|_{\Hd\to \Hd}\geq \int_0^\infty t^{1/p-1} \varphi(t)dt - 2 \int_m^\infty t^{1/p-1} \varphi(t)dt \to \int_0^\infty t^{1/p-1} \varphi(t)dt$$
as $m\to\infty$ since $\int_0^\infty t^{1/p-1} \varphi(t)dt<\infty$. This proves (\ref{main theorem, 0}).

\vskip 0.3cm

Now we return to prove (\ref{necessary condition}). Indeed, we consider the following two cases.

{\bf Case 1:} $p=1$. Take $f(z)= \frac{1}{(z+i)^2}$ for all $z\in\mathbb C_+$. Then
$$\|f\|_{\mathcal H^1_a(\mathbb C_+)}= \int_{-\infty}^{\infty} \frac{1}{x^2+1} dx<\infty.$$
Therefore, by the Fatou lemma, we get  
\begin{eqnarray*}
	\infty > \|\Hau f\|_{\mathcal H^1_a(\mathbb C_+)} &=& \sup_{y>0} \int_{-\infty}^{\infty} \left|\int_0^\infty \frac{1}{\left[\frac{x}{t} + i\left(\frac{y}{t}+1\right)\right]^2} \frac{\varphi(t)}{t}dt\right|dx\\
	&\geq& 2  \sup_{y>0} \int_{0}^{\infty} dx \int_{0}^{\infty}\frac{\frac{x}{t}\left(\frac{y}{t}+1\right)}{\left[\left(\frac{x}{t}\right)^2 + \left(\frac{y}{t}+1\right)^2\right]^2} \frac{\varphi(t)}{t}dt\\
	&\geq& 2 \int_{0}^{\infty} dx \int_{0}^{\infty} \frac{\frac{x}{t}}{\left[\left(\frac{x}{t}\right)^2 + 1\right]^2} \frac{\varphi(t)}{t}dt\\
	&=& 2 \int_{0}^{\infty}\frac{u}{[u^2+1]^2} du \int_0^\infty \varphi(t)dt.
\end{eqnarray*}
 This proves (\ref{necessary condition}).

\vskip 0.15cm

{\bf Case 2:} $1<p<\infty$. For any $0<\varepsilon<1-1/p$, take
$$f_\varepsilon(z)= \left(\frac{1}{z+ i \varepsilon }\right)^{1/p+\varepsilon}$$
for all $z\in \mathbb C_+$. Then
\begin{equation}\label{norm of f for p>1}
\|f_\varepsilon\|_{\mathcal H^p_a(\mathbb C_+)}= \left(\int_{-\infty}^{\infty} \frac{1}{{\sqrt{x^2+\varepsilon^2}}^{1+p\varepsilon}} dx\right)^{1/p} <\infty
\end{equation}
and
\begin{eqnarray*}
	\infty>\|\Hau(f_\varepsilon)\|^p_{\mathcal H^p_a(\mathcal C_+)} &=& \sup_{y>0}\int_{-\infty}^{\infty}\left|\int_0^\infty \left(\frac{1}{\frac{x}{t}+ i\left(\frac{y}{t}+\varepsilon\right)}\right)^{1/p+\varepsilon} \frac{\varphi(t)}{t} dt \right|^p dx\\
	&\geq& \int_0^\infty \left|\int_0^\infty \frac{\frac{x}{t}}{\sqrt{\left(\frac{x}{t}\right)^2+\varepsilon^2}} \frac{1}{\sqrt{\left(\frac{x}{t}\right)^2+ \varepsilon^2}^{1/p+\varepsilon}}\frac{\varphi(t)}{t} dt\right|^p dx,
\end{eqnarray*}
where we used the Fatou lemma and the fact that
$$\mbox{Re} \left(\frac{1}{\frac{x}{t}+ i\left(\frac{y}{t}+\varepsilon\right)}\right)^{1/p+\varepsilon}\geq \frac{\frac{x}{t}}{\sqrt{\left(\frac{x}{t}\right)^2+\left(\frac{y}{t}+\varepsilon\right)^2}} \frac{1}{\sqrt{\left(\frac{x}{t}\right)^2+\left(\frac{y}{t}+\varepsilon\right)^2}^{1/p+\varepsilon}}$$
for all $x,y,t>0$ since $0<1/p+\varepsilon<1$. This, together with (\ref{norm of f for p>1}), gives
\begin{eqnarray*}
	\|\Hau\|^p_{\mathcal H^p_a(\mathbb C_+)\to \mathcal H^p_a(\mathbb C_+)} &\geq& \frac{\|\Hau(f_\varepsilon)\|^p_{\mathcal H^p_a(\mathbb C_+)}}{\|f_\varepsilon\|^p_{\mathcal H^p_a(\mathbb C_+)}}\\
	&\geq& \frac{\int_1^\infty \left|\int_0^{1/\varepsilon} \frac{\frac{x}{t}}{\sqrt{\left(\frac{x}{t}\right)^2+\varepsilon^2}} \frac{1}{\sqrt{\left(\frac{x}{t}\right)^2+ \varepsilon^2}^{1/p+\varepsilon}}\frac{\varphi(t)}{t} dt\right|^p dx}{2 \int_0^\infty \frac{1}{{\sqrt{x^2+\varepsilon^2}}^{1+p\varepsilon}} dx}\\
	&\geq& \frac{1}{2^{\frac{3+ p(1+\varepsilon)}{2}}} \left(\int_{0}^{1/\varepsilon} t^{1/p-1+\varepsilon} \varphi(t)dt\right)^p \frac{\int_1^\infty \frac{1}{x^{1+p\varepsilon}}dx}{\varepsilon^{-p\varepsilon}\int_0^\infty \frac{1}{\sqrt{x^2+1}^{1+p\varepsilon}}dx}.
\end{eqnarray*}
Hence,
$$\int_{0}^{1/\varepsilon} t^{1/p-1+\varepsilon} \varphi(t)dt \leq 2^{\frac{3+ p(1+\varepsilon)}{2p}} \varepsilon^{-\varepsilon} \left(\frac{\int_0^\infty \frac{1}{\sqrt{x^2+1}^{1+p\varepsilon}}dx}{\int_1^\infty \frac{1}{x^{1+p\varepsilon}}dx}\right)^{1/p} \|\Hau\|_{\mathcal H^p_a(\mathbb C_+)\to \mathcal H^p_a(\mathbb C_+)}.$$
Letting $\varepsilon \to 0$, we obtain 
$$
\int_0^\infty t^{1/p-1} \varphi(t)dt \leq 2^{\frac{3+p}{2p}}\|\Hau\|_{\mathcal H^p_a(\mathbb C_+)\to \mathcal H^p_a(\mathbb C_+)}<\infty.
$$
This proves (\ref{necessary condition}), and thus ends the proof of Theorem \ref{main theorem}.

\end{proof}


\section{Some applications}

Let $1\leq p<\infty$, we define (see \cite{St}) the {\it Hilbert transform} of $f\in L^p(\R)$ by
$$H(f)(x):= \frac{1}{\pi}{\rm p. v.}\int_{-\infty}^{\infty} \frac{f(y)}{x-y}dy, \quad x\in\R.$$

\begin{theorem}\label{commutes with the Hilbert transform}   
	Let $p\in (1,\infty)$ and let $\varphi$ be as in Theorem \ref{main theorem}. Then $H_\varphi$ is bounded on $L^p(\R)$ if and only if (\ref{main inequality}) holds. Moreover, in that case, 
	$$\|H_\varphi\|_{L^p(\R)\to L^p(\R)}= \int_0^\infty t^{1/p-1}\varphi(t)dt$$
	and $H_\varphi$ commutes with the Hilbert transform $H$ on $L^p(\R)$.    
\end{theorem}

In order to prove Theorem \ref{commutes with the Hilbert transform}, we need the following lemmas.

\begin{lemma}[see \cite{Du, Ga}]\label{boundary value characterization 1}
	Let $1< p<\infty$. Then:
	\begin{enumerate}[\rm (i)]
		\item If $g\in L^p(\R)$, then $f^*:= g + iH(g)$ is the boundary value function of some function $f\in \mathcal H_a^p(\mathbb C_+)$.
		\item Conversely, if $f^*$ is a boundary value function of $f\in \mathcal H_a^p(\mathbb C_+)$, then there exists a real-valued function $g\in L^p(\R)$ such that $f^*= g+iH(g)$.
	\end{enumerate}
	Moreover, in those cases,
	$$\|g\|_{L^p(\R)} \sim \|g+ iH(g)\|_{L^p(\R)}=\|f^*\|_{L^p(\R)}=\|f\|_{\mathcal H_a^p(\mathbb C_+)}.$$
\end{lemma}

\begin{lemma}[see \cite{An, Xi}]\label{key lemma for commuting relation with the Hilbert transform}
	Let $p\in (1,\infty)$ and let $\varphi$ be such that (\ref{main inequality}) holds. Then:
	\begin{enumerate}[\rm (i)]
		\item $H_\varphi$ is bounded on $L^p(\R)$, moreover,
		$$\|H_\varphi\|_{L^p(\R)\to L^p(\R)}\leq \int_0^\infty t^{1/p-1}\varphi(t)dt.$$
		\item If supp $\varphi\subset [0,1]$, then
		$$\|H_\varphi\|_{L^p(\R)\to L^p(\R)}= \int_0^1 t^{1/p-1}\varphi(t)dt.$$
	\end{enumerate}
\end{lemma}

\begin{proof}[Proof of Theorem \ref{commutes with the Hilbert transform}]
	Suppose that (\ref{main inequality}) holds. By Lemma \ref{key lemma for commuting relation with the Hilbert transform}(i),
	\begin{equation}\label{commutes with the Hilbert transform, 1}
	\|H_\varphi\|_{L^p(\R)\to L^p(\R)}\leq \int_0^\infty t^{1/p-1}\varphi(t)dt.
	\end{equation}
	
	Conversely, suppose that $H_\varphi$ is bounded on $L^p(\R)$. We first claim that
	\begin{equation}\label{commutes with the Hilbert transform, 2}
	\int_0^\infty t^{1/p-1}\varphi(t)dt<\infty.
	\end{equation}
	
	Assume (\ref{commutes with the Hilbert transform, 2}) holds for a moment.
	
	For any $m>0$,  take $\varphi_m$ is as in the proof of Theorem \ref{main theorem}. Then, by a similar argument to the proof of Theorem \ref{main theorem}, we get
	\begin{eqnarray*}
		\|H_\varphi\|_{L^p(\R)\to L^p(\R)} &\geq& \|H_{\varphi_m\left(\frac{\cdot}{m}\right)}\|_{L^p(\R)\to L^p(\R)} - \|H_\varphi - H_{\varphi_m\left(\frac{\cdot}{m}\right)}\|_{L^p(\R)\to L^p(\R)}\\
		&\geq& \int_0^\infty t^{1/p-1}\varphi(t)dt - 2 \int_m^\infty t^{1/p-1}\varphi(t)dt\to \int_0^\infty t^{1/p-1}\varphi(t)dt
	\end{eqnarray*}
	as $m\to\infty$. This, together with (\ref{commutes with the Hilbert transform, 1}), yields
	$$\|H_\varphi\|_{L^p(\R)\to L^p(\R)} = \int_0^\infty t^{1/p-1}\varphi(t)dt.$$
	
	Now let us return to prove (\ref{commutes with the Hilbert transform, 2}). Indeed, for any $\epsilon\in (0,1)$, take
	$$f_\epsilon(x)= |x|^{-1/p-\epsilon}\chi_{\{y\in\R: |y|>1\}}(x)$$
	and
	$$g_\epsilon(x)= |x|^{-1/p+\epsilon}\chi_{\{y\in\R: |y|<1\}}(x)$$
	for all $x\in\R$. Then some simple computations give
	$$\|H_\varphi\|_{L^p(\R)\to L^p(\R)}\geq \frac{\|H_\varphi(f_\epsilon)\|_{L^p(\R)}}{\|f_\epsilon\|_{L^p(\R)}} \gtrsim \epsilon^\epsilon \int_1^{1/\epsilon} t^{1/p-1}\varphi(t)dt$$
	and
	$$\|H_\varphi\|_{L^p(\R)\to L^p(\R)}\geq \frac{\|H_\varphi(g_\epsilon)\|_{L^p(\R)}}{\|g_\epsilon\|_{L^p(\R)}} \gtrsim \epsilon^\epsilon \int_\epsilon^{1} t^{1/p-1}\varphi(t)dt.$$
	Letting $\epsilon\to 0$, we get
	$$\int_1^\infty t^{1/p-1}\varphi(t)dt \lesssim \|H_\varphi\|_{L^p(\R)\to L^p(\R)}<\infty$$
	and
	$$\int_0^1 t^{1/p-1}\varphi(t)dt \lesssim \|H_\varphi\|_{L^p(\R)\to L^p(\R)}<\infty.$$
	This proves (\ref{commutes with the Hilbert transform, 2}).
	
	\vskip 0.2cm
	
	Finally, we need to show that $H_\varphi$ commutes with the Hilbert transform $H$ on $L^p(\R)$. To this ends, it suffices to show
	\begin{equation}\label{commutes with the Hilbert transform, 3}
	H_\varphi(H(f))= H(H_\varphi(f))
	\end{equation}
	for all real-valued functions $f$ in $L^p(\R)$. Indeed, by Theorem \ref{main theorem} and Lemma \ref{boundary value characterization 1}, there exists a real-valued function $g$ in $L^p(\R)$ such that
	$$g + i H(g)= H_\varphi(f+ i H(f)).$$
	This proves (\ref{commutes with the Hilbert transform, 3}), and thus completes the proof of Theorem \ref{commutes with the Hilbert transform}.
\end{proof}

Let $1<p<\infty$, we denote by $H^p_+(\R)$ and $H^p_-(\R)$ the subspaces of $L^p(\R)$ consisting of those functions whose Poisson extensions to the upper half-plane $\mathbb C_+$ are holomorphic and anti-holomorphic, respectively. 

It is well-known (see \cite{Du, Ga, St}) that
\begin{equation}\label{upper Hardy spaces}
	H^p_{+}(\R)= \{f+ iH(f): f\in L^p(\R)\}
\end{equation}
and
\begin{equation}\label{lower Hardy spaces}
H^p_{-}(\R)= \{f- iH(f): f\in L^p(\R)\}.
\end{equation}
Moreover, $L^p(\R)= H^p_{+}(\R) \oplus H^p_{-}(\R)$.

\begin{theorem}\label{first corollary}
	Let $p\in (1,\infty)$ and let $\varphi$ be such that (\ref{main inequality}) holds. Then $H_\varphi$ is bounded on the space $H^p_+(\R)$, moreover,
	$$\|H_\varphi\|_{H^p_+(\R)\to H^p_+(\R)}= \int_0^\infty t^{1/p-1}\varphi(t)dt$$
	and $H_\varphi$ commutes with the Hilbert transform $H$ on $H^p_+(\R)$.
\end{theorem}

\begin{proof}
	It follows from Theorem \ref{commutes with the Hilbert transform} that $H_\varphi(f)$ belongs to $H^p_+(\R)$ for all $f\in H^p_+(\R)$, and thus
	\begin{equation}\label{a corollary, 1}
	\|H_\varphi\|_{H^p_+(\R)\to H^p_+(\R)}\leq \|H_\varphi\|_{L^p(\R)\to L^p(\R)}= \int_0^\infty t^{1/p-1}\varphi(t)dt.
	\end{equation}
	
	For any $\varepsilon >0$, by Theorem \ref{main theorem}, there exists $f_\varepsilon\in \Hd$ for which
	$$\frac{\|H_\varphi(f_\varepsilon^*)\|_{L^p(\R)}}{\|f_\varepsilon^*\|_{L^p(\R)}}= \frac{\|(\Hau f_\varepsilon)^*\|_{L^p(\R)}}{\|f_\varepsilon^*\|_{L^p(\R)}} = \frac{\|\Hau f_\varepsilon\|_{\Hd}}{\|f_\varepsilon\|_{\Hd}}\geq \int_0^\infty t^{1/p-1}\varphi(t)dt -\varepsilon.$$
	This, together with (\ref{a corollary, 1}), allows us to conclude that
	$$\|H_\varphi\|_{H^p_+(\R)\to H^p_+(\R)}=\int_0^\infty t^{1/p-1}\varphi(t)dt.$$
	
	Finally, $H_\varphi$ commutes with the Hilbert transform $H$ on $H^p_+(\R)$ is followed from Theorem \ref{commutes with the Hilbert transform} and (\ref{upper Hardy spaces}).
\end{proof}

\begin{theorem}\label{second corollary}
	Let $p\in (1,\infty)$ and let $\varphi$ be such that (\ref{main inequality}) holds. Then $H_\varphi$ is bounded on the space $H^p_-(\R)$, moreover,
	$$\|H_\varphi\|_{H^p_-(\R)\to H^p_-(\R)}= \int_0^\infty t^{1/p-1}\varphi(t)dt$$
	and $H_\varphi$ commutes with the Hilbert transform $H$ on $H^p_-(\R)$.
\end{theorem}

\begin{proof}
	It follows from Theorem \ref{first corollary} and the fact that $f\in H^p_+(\R)$ if and only if $\bar f\in H^p_-(\R)$.
\end{proof}

Let $\Phi$ be in the Schwartz space $\mathcal S(\R)$ satisfying $\int_{\R}\Phi(x)dx\ne 0$. For any $t>0$, set $\Phi_t(x):= t^{-1}\Phi(x/t)$. Following Fefferman and Stein \cite{FS, St}, we define the {\it real Hardy space} $\H^1(\R)$ as the set of all functions $f\in L^1(\R)$ such that 
$$\|f\|_{\H^1(\R)}:= \left\|M_{\Phi}(f)\right\|_{L^1(\R)}<\infty,$$
where $M_{\Phi}(f)$ is the {\it smooth maximal function} of $f$ defined by
$$M_{\Phi}(f)(x)= \sup_{t>0}|f*\Phi_t(x)|,\quad x\in\R.$$

Remark that the norm $\|\cdot\|_{\H^1(\R)}$ depends on the choice of $\Phi$, but the space $\H^1(\R)$ does not depend on this choice (see Proposition \ref{some equivalent characterizations of H1} below). 

The following lemma is well-known.

\begin{lemma}[see \cite{Du, Ga, QXYYY}]\label{boundary value characterization}
	\begin{enumerate}[\rm (i)]
		\item If $g\in \H^1(\R)$, then $f^*:= g + iH(g)$ is the boundary value function of some function $f\in \mathcal H_a^1(\mathbb C_+)$.
		\item Conversely, if $f^*$ is a boundary value function of $f\in \mathcal H_a^1(\mathbb C_+)$, then there exists a real-valued function $g\in \H^1(\R)$ such that $f^*= g+iH(g)$.
	\end{enumerate}
	Moreover, in those cases,
	$$\|g\|_{\H^1(\R)} \sim \|g+ iH(g)\|_{\H^1(\R)}=\|f^*\|_{\H^1(\R)}\sim \|f^*\|_{L^1(\R)}=\|f\|_{\mathcal H_a^1(\mathbb C_+)}.$$
\end{lemma}

Let $P_t(x)=\frac{1}{\pi}\frac{t}{x^2+t^2}$ be the Poisson kernel on $\R$. For any $f\in L^1(\R)$, we denote $u(y,t)= f*P_t(y)$. Then, set
$$\mathcal M_P(f)(x)= \sup_{|y-x|<t} |u(y,t)|\quad\mbox{and}\quad S(f)(x)=\left[\iint_{|y-x|<t}\left(|u_t(y,t)|^2 +|u_y(y,t)|^2\right) dy dt\right]^{1/2}.$$

A function $a$ is called an $\H^1$-atom related to the interval $B$ if
\begin{enumerate}[$\bullet$]
	\item supp $a\subset B$;
	\item $\|a\|_{L^\infty(\R)}\leq |B|^{-1}$;
	\item $\int_{\R} a(x)dx=0$.
\end{enumerate} 
We define the Hardy space $\H^1_{\rm at}(\R)$ as the space of functions $f\in L^1(\R)$ which can be written as $f=\sum_{j=1}^\infty \lambda_j a_j$ with $a_j$'s are $\H^1$-atoms and $\lambda_j$'s are complex numbers satisfying $\sum_{j=1}^\infty |\lambda_j|<\infty$. The norm on $\H^1_{\rm at}(\R)$ is then defined by
$$\|f\|_{\H^1_{\rm at}(\R)}:= \inf\left\{\sum_{j=1}^\infty |\lambda_j|: f= \sum_{j=1}^\infty \lambda_j a_j\right\}.$$ 

The following proposition is classical and can be found in Stein's book \cite{St}.

\begin{proposition}\label{some equivalent characterizations of H1}
	Let $f\in L^1(\R)$. Then the following conditions are equivalent:
	\begin{enumerate}[\rm (i)]
		\item $f\in \H^1(\R)$.
		\item $\mathcal M_\Phi(f)(\cdot)= \sup_{|y-\cdot|<t}|f*\Phi_t(y)|\in L^1(\R)$.
		\item $\mathcal M_P(f)\in L^1(\R)$.
		\item $S(f)\in L^1(\R)$.
		\item $f\in \H^1_{\rm at}(\R)$.
		\item $H(f)\in L^1(\R)$.
	\end{enumerate}
	Moreover, in those cases,
	\begin{eqnarray*}
		\|f\|_{\H^1(\R)} &\sim& \|\mathcal M_\Phi(f)\|_{L^1(\R)}\sim \|\mathcal M_P(f)\|_{L^1(\R)}\sim \|S(f)\|_{L^1(\R)}\\
		&\sim & 	\|f\|_{\H^1_{\rm at}(\R)}\sim \|f\|_{L^1(\R)} + \|H(f)\|_{L^1(\R)}.
	\end{eqnarray*}
	Of course, the above constants are depending on $\Phi$.
\end{proposition}

The following gives a lower bound for the norm of $H_\varphi$ on $\H^1(\R)$.

\begin{theorem}\label{a lower bound}
	Let $\|\cdot\|_*$ be one of the six norms in Proposition \ref{some equivalent characterizations of H1}. Assume that $H_\varphi$ is bounded on $(\H^1(\R),\|\cdot\|_*)$. Then,
	$$\int_0^\infty \varphi(t)dt \leq \|H_\varphi\|_{(\H^1(\R),\|\cdot\|_*)\to (\H^1(\R),\|\cdot\|_*)}<\infty$$
	and $H_\varphi$ commutes with the Hilbert transform $H$ on $\H^1(\R)$.
\end{theorem}

It should be pointed out that, when supp $\varphi\subset [1,\infty)$ and $\|\cdot\|_{*}=\|\cdot\|_{\H^1_{\rm at}(\R)}$, the above theorem is due to Xiao \cite[p. 666]{Xi} (see also \cite{LL, Li08}).

In order to prove Theorem \ref{a lower bound}, we need the following lemma.

\begin{lemma}\label{a key lemma for real Hausdorff operators}
	Let $\varphi$ be such that $\int_0^\infty \varphi(t)dt<\infty$ and supp $\varphi\subset [0,1]$. Then,
	$$\int_0^1 \varphi(t)dt\leq\|H_\varphi\|_{(\H^1(\R),\|\cdot\|_*)\to (\H^1(\R),\|\cdot\|_*)}<\infty.$$ 
\end{lemma}

\begin{proof}
	It is well-known  (see \cite{An, HKQ, LM}) that if $\int_0^\infty \varphi(t)dt<\infty$, then $H_\varphi$ is bounded on $\H^1(\R)$, moreover,
	\begin{equation}\label{an upper bound for the norm}
		\|H_\varphi\|_{(\H^1(\R),\|\cdot\|_*)\to (\H^1(\R),\|\cdot\|_*)} \lesssim \int_0^\infty \varphi(t)dt= \int_0^1 \varphi(t)dt.
	\end{equation}

	We now show that
	$$\int_0^1 \varphi(t)dt\leq\|H_\varphi\|_{(\H^1(\R),\|\cdot\|_*)\to (\H^1(\R),\|\cdot\|_*)}.$$
	Indeed, let $\delta\in (0,1)$ and $\varphi_\delta$ be as in the proof of Lemma \ref{key lemma}(ii). For any $\varepsilon>0$, define the function $f_\varepsilon: \mathbb C_+\to\mathbb C$ by
	$$f_\varepsilon(z)=\frac{1}{(z+i)^{1+\varepsilon}}.$$
	Then, by Lemma \ref{key lemma}(iii), Lemma \ref{boundary value characterization}, Proposition \ref{some equivalent characterizations of H1} and (\ref{an estimate for the norm}), 
	\begin{eqnarray*}
	&& \frac{\left\|H_{\varphi_\delta}(f_\varepsilon^*) - f_\varepsilon^* \int_0^\infty \varphi_\delta(t)dt \right\|_{*}}{\|f_\varepsilon^*\|_{*}}\\
	&\lesssim& \frac{\left\|\mathscr H_{\varphi_\delta}(f_\varepsilon) - f_\varepsilon \int_0^\infty \varphi_\delta(t)dt \right\|_{\mathcal H^1_a(\mathbb C_+)}}{\|f_\varepsilon\|_{\mathcal H^1_a(\mathbb C_+)}}\to 0
	\end{eqnarray*}
	as $\varepsilon\to 0$. This implies that
	$$\int_{\delta}^{1} \varphi(t)dt= \int_{0}^{\infty} \varphi_\delta(t)dt\leq \|H_{\varphi_\delta}\|_{(\H^1(\R),\|\cdot\|_*)\to (\H^1(\R),\|\cdot\|_*)},$$
	and thus
	$$\int_0^1 \varphi(t)dt\leq \|H_{\varphi}\|_{(\H^1(\R),\|\cdot\|_*)\to (\H^1(\R),\|\cdot\|_*)}$$
	since 
	$$\|H_{\varphi}- H_{\varphi_\delta}\|_{(\H^1(\R),\|\cdot\|_*)\to (\H^1(\R),\|\cdot\|_*)}\lesssim \int_0^\infty (\varphi(t)-\varphi_\delta(t))dt= \int_0^\delta \varphi(t)dt\to 0$$
	 as $\delta \to 0$. This ends the proof of Lemma \ref{a key lemma for real Hausdorff operators}.	
\end{proof}

\begin{proof}[Proof of Theorem \ref{a lower bound}]
	It follows from \cite[Theorem 3.3]{HKQ} that 
	$$\int_0^\infty \varphi(t)dt <\infty.$$
	
	For any $m>0$, let $\varphi_m$ be as in the proof of Theorem \ref{main theorem}. Then, by (\ref{an upper bound for the norm}),
	\begin{eqnarray}\label{approximates}
		&&\left\|H_\varphi - H_{\varphi_m(\frac{\cdot}{m})}\right\|_{(\H^1(\R),\|\cdot\|_*)\to (\H^1(\R),\|\cdot\|_*)} \\
		&=& \left\|H_{\varphi- \varphi_m(\frac{\cdot}{m})}\right\|_{(\H^1(\R),\|\cdot\|_*)\to (\H^1(\R),\|\cdot\|_*)}\nonumber\\
		&\lesssim& \int_0^\infty \left[\varphi(t)-\varphi_m(\frac{t}{m})\right]dt =\int_m^\infty \varphi(t)dt,\nonumber
	\end{eqnarray}
	where the constant is independent of $m$. 
	
	Noting that
	$$\left\|f\left(\frac{\cdot}{m}\right)\right\|_{*}= m \|f(\cdot)\|_{*}\quad\mbox{and}\quad  H_{\varphi_m\left(\frac{\cdot}{m}\right)} f= H_{\varphi_m} f\left(\frac{\cdot}{m}\right)$$
	for all $f\in \H^1(\R)$, Lemma \ref{a key lemma for real Hausdorff operators} gives
	\begin{eqnarray*}
		\left\|H_{\varphi_m\left(\frac{\cdot}{m}\right)}\right\|_{(\H^1(\R),\|\cdot\|_*)\to (\H^1(\R),\|\cdot\|_*)} &=& m \left\|H_{\varphi_m}\right\|_{(\H^1(\R),\|\cdot\|_*)\to (\H^1(\R),\|\cdot\|_*)}\\
		&\geq& m \int_0^1  \varphi_m(t)dt = \int_0^m  \varphi(t)dt.
	\end{eqnarray*}
	This, together with (\ref{approximates}) and $\lim\limits_{m\to\infty} \int_m^\infty \varphi(t)dt=0$, allows us to conclude that
	$$\int_0^\infty \varphi(t)dt\leq \left\|H_\varphi\right\|_{(\H^1(\R),\|\cdot\|_*)\to (\H^1(\R),\|\cdot\|_*)}.$$
	
	Using the Fourier transform, Liflyand and M\'oricz proved in \cite{LM2} that $H_\varphi$ commutes with the Hilbert transform $H$ on $\H^1(\R)$. However, we also would like to give a new proof of this fact here. It suffices to prove 
	\begin{equation}\label{commutes with the Hilbert transform at p=1}
		H_\varphi(H(f))= H(H_\varphi(f))
	\end{equation}
	for all real-valued functions $f$ in $\H^1(\R)$. Indeed, by Theorem \ref{main theorem} and Lemma \ref{boundary value characterization}, there exists a real-valued function $g$ in $\H^1(\R)$ such that
	$$g + i H(g)= H_\varphi(f+ i H(f)).$$
	This proves (\ref{commutes with the Hilbert transform at p=1}), and thus completes the proof of Theorem \ref{a lower bound}.

\end{proof}

Let $a: (0,\infty) \to [0,\infty)$ be a measurable function. Following  Carro and Ortiz-Caraballo \cite{CO}, we define
$$\mathscr S_a F(z)=\int_0^\infty F(tz)a(t)dt,\quad z\in \mathbb C_+,$$
for all holomorphic functions $F$ on $\mathbb C_+$; and define
$$S_a f(x)=\int_0^\infty f(tx) a(t)dt, \quad x\in \R,$$
for all measurable functions $f$ on $\R$.

It is easy to see that
$$\mathscr S_a F= \H_{\varphi}F \quad\mbox{and}\quad S_a f= H_\varphi f,$$
where $\varphi(t)= t^{-1}a(t^{-1})$ for all $t\in (0,\infty)$. Hence, it follows from Theorems \ref{main theorem}, \ref{commutes with the Hilbert transform} and \ref{a lower bound} that:

\begin{theorem}\label{a theorem for adjointed operators}
	Let $p\in [1,\infty]$ and let  $a: (0,\infty) \to [0,\infty)$ be a measurable function. Then $\mathscr S_a$ is bounded on $\H_a^p(\mathbb C_+)$ if and only if 
	\begin{equation}\label{an inequality for adjoint operators}
	\int_0^\infty t^{-1/p} a(t)dt<\infty.
	\end{equation}
	Moreover, when (\ref{an inequality for adjoint operators}) holds, we obtain
	$$
	\|\mathscr S_a\|_{\H^p_a(\mathbb C_+)\to \H^p_a(\mathbb C_+)}=\int_0^\infty t^{-1/p} a(t)dt
	$$
	and, for any $f\in \H_a^p(\mathbb C_+)$,
	$$
	(\mathscr S_a f)^*= S_a(f^*).
	$$
\end{theorem}

\begin{theorem}\label{commutes with the Hilbert transform for adjointed operator}
	Let $p\in (1,\infty)$ and let $a$ be as in Theorem \ref{a theorem for adjointed operators}. Then $S_a$ is bounded on $L^p(\R)$ if and only if (\ref{an inequality for adjoint operators}) holds. Moreover, in that case, 
	$$\|S_a\|_{L^p(\R)\to L^p(\R)}= \int_0^\infty t^{-1/p} a(t)dt$$
	and $S_a$ commutes with the Hilbert transform $H$ on $L^p(\R)$.    
\end{theorem}

\begin{theorem}\label{commutes with the Hilbert transform on real Hardy space for adjointed operator}
	Let $a$ be as in Theorem \ref{a theorem for adjointed operators}. Then $S_a$ is bounded on $\H^1(\R)$ if and only if $\int_0^\infty t^{-1} a(t)dt<\infty.$
	Moreover, in that case,
	$$\int_0^\infty t^{-1} a(t)dt\leq \|S_a\|_{\H^1(\R)\to \H^1(\R)}<\infty$$
	and $S_a$ commutes with the Hilbert transform $H$ on $\H^1(\R)$.
\end{theorem}

Also it is easy to see that if (\ref{main inequality}) holds for $1<p<\infty$, then
$$\int_{\R} H_\varphi f(x) g(x)dx= \int_{\R} f(x) S_\varphi g(x)dx$$
whenever $f\in L^p(\R)$ and $g\in L^q(\R)$, $q=p/(p-1)$. Namely, $S_\varphi$ can be viewed as the Banach space adjoint of $H_\varphi$ and vice versa. Therefore, by Theorem \ref{a lower bound}, Theorem \ref{commutes with the Hilbert transform on real Hardy space for adjointed operator} and \cite[Theorem 1]{LL}, a duality argument gives:

\begin{theorem}
	\begin{enumerate}[\rm (i)]
		\item $S_\varphi$ is bounded on $BMO(\R)$ if and only if $\int_0^\infty \varphi(t)dt<\infty$. Moreover, in that case,
		$$\|S_\varphi\|_{BMO(\R)\to BMO(\R)}=\int_0^\infty \varphi(t)dt.$$
		\item $H_\varphi$ is bounded on $BMO(\R)$ if and only if $\int_0^\infty t^{-1}\varphi(t)dt<\infty$. Moreover, in that case,
		$$\|H_\varphi\|_{BMO(\R)\to BMO(\R)}=\int_0^\infty t^{-1}\varphi(t)dt.$$
	\end{enumerate}
\end{theorem}

Here the space $BMO(\R)$ (see \cite{FS, JN}) is the dual space of $\H^1(\R)$ defined as the space of all functions $f\in L^1_{\rm loc}(\R)$ such that
$$\|f\|_{BMO(\R)}:=\sup_{B}\frac{1}{|B|}\int_B \left|f(x)-\frac{1}{|B|}\int_B f(y)dy\right|dx<\infty,$$
where the supremum is taken over all intervals $B\subset\R$. 

\vskip 0.5cm

{\bf Acknowledgements.}  The authors would like to thank the referees for their
carefully reading and helpful suggestions.

\end{document}